\documentclass{article}

\usepackage{amsmath,amssymb,amsthm,graphics,graphicx}
\usepackage[left=3cm]{geometry}
\usepackage{hyperref}
\usepackage[dvipsnames]{xcolor}
\newcommand{\comment}[1]{}

\newtheorem{thm}{Theorem}[section]
\newtheorem{cor}[thm]{Corollary}
\newtheorem{lem}[thm]{Lemma}

\newcommand{\be}{\begin{equation}}
\newcommand{\eq}{\end{equation}}
\newcommand{\la}{\langle}
\newcommand{\ra}{\rangle}
\newcommand{\U}{U_{6n}}
\usepackage{eqnarray}

\date{}
\begin{document}
	\title{On the Non-Commuting Graph of the Group $U_{6n}$}
	\author{Sanhan Khasraw$^{1}$, C.H. Jaf$^{2}$, Nor Haniza Sarmin$^{3}$, Ibrahim Gambo$^{3}$\\
	$ ^{1}$ Department of Mathematics, College of Basic Education,\\ Salahadddin University-Erbil, Erbil, Iraq.\\ sanhan.khasraw@su.edu.krd\\
	$ ^{2}$ Independent Researcher.\\ cah.barrami@gmail.com\\
	$^{3}$ Department of Mathematical Sciences, Faculty of Science,\\ Universiti Teknologi, Malaysia, 81310 UTM Johor Bahru, Malaysia.\\
	nhs@utm.my, igambo@utm.my\\
		}
	
	\maketitle

	\begin{abstract}
	    A non-commuting graph of a finite group $G$ is a graph whose vertices are non-central elements of $G$ and two vertices are adjacent if they don't commute in $G$. In this paper, we study the non-commuting graph of the group $\U$ and explore some of its properties including the independent number, clique and chromatic numbers. Also, the general formula of the resolving polynomial of the non-commuting graph of the group $U_{6n}$ are provided. Furthermore, we find the detour index, eccentric connectivity, total eccentricity and independent polynomials of the graph.    
	\end{abstract}
	\hspace{1cm}\\
	\text{\bf Keywords}: Non-commuting graph, independent number, chromatic number, clique number, resolving polynomial of a graph.
	
	\section{Introduction}
	In the last three decades, the interrelation of the structure in graphs and algebra has provided us some interesting results and the topic has earned significant attention from the research community, for example see \cite{bondy1976graph, mirzargar2012some}.
	
	Suppose that $G$ is a finite group. We denote the center of $G$ by $Z(G)$ and the centralizer of an element $a$ in $G$ by $C_G(a)$. If $C_G(a)$ is abelian for every $a \in G$, then $G$ is said to be an \textit{AC-group} \cite{mirzargar2012some}. The \textit{non-commuting graph} $\Gamma(G)$ of $G$ has $G-Z(G)$ as its vertex set in which two vertices are adjacent if they don't commute in $G$. The Hungarian mathematician Paul Erdos was the first to introduce the concept of non-commuting graph of a group in the 20th century. Since then, the topic has been widely studied by researchers in the field, see \cite{abdollahi2006non, erdosproblem, vatandoost2018domination}. It is worth-noting that if the underlying group is abelian then the non-commuting graph has no element since in that case $G=Z(G)$. In this paper we assume $G$ to be the non-abelian group $U_{6n}$, which will be defined later.	
	As a matter of fact there are some researches studying commuting graph of a group, for instance \cite{ali2016commuting}, and there are several articles about the non-commuting graph of a group, for example, Abdollahi \textit{et al.} \cite{abdollahi2006non} investigated on the non-commuting graph of finite groups whereas Talebi \cite{Talebi} has conducted the same investigation for the dihedral groups. 
	
	Most recently, Khasraw \textit{et al.} \cite{Skhasraw} investigated the detour index, eccentric connectivity, total eccentricity polynomials and the mean distance of the non-commuting graph of dihedral group.
		
	Some fundamental concepts that are related to this research throughout this paper are provided in what follows. The graphs considered in this paper are simple, that is, undirected with no loops or multiple edges, and a graph is \textit{finite} if both its vertex and edge sets are finite. Hence, we denote the vertex and edge sets of the graph $\Gamma$ respectively by $V(\Gamma)$ and $E(\Gamma)$ while we denote the number of vertices in the graph $\Gamma$ by $n(\Gamma)$ and the number of its edges by $|E(\Gamma)|$.  
	
	Let $k \ge 2$. A sequence of $k$ vertices in which each vertex in the sequence is adjacent to a vertex next to it is known as a \textit{path} of a graph, denoted by $P_k$. A path that does not repeat vertices is called a \textit{simple path}. A path that starts and ends at the same vertex is referred to us \textit{circuit} while any circuit that does not repeat vertices is a \textit{cycle}, denoted as $C_n$, where $n \ge 3$ \cite{bondy1976graph, Die}. A graph is said to be \textit{connected} if for every pair of vertices $u$ and $v$, there is a path from $u$ to $v$, while a \textit{disconnected} graph consists of connected pieces called components. If vertices $u$ and $v$ are connected in $\Gamma$, the distance(detour distance) between $u$ and $v$, denoted by $d(u, v)(D(u, v))$, is the length of a shortest(longest) $u-v$ path in $\Gamma$. For a given vertex $v$ in $\Gamma$, the maximum distance between $v$ and any other vertex in $\Gamma$ is called the \textit{eccentricity} of $v$, denoted by $ecc(v)$. The \textit{degree} $\deg(v)$ of a vertex $v$ is the number of edges incident with $v$ \cite{bondy1976graph}. 
	
	For a graph $\Gamma$, the polynomials $D(\Gamma, x)=\sum_{u,v \in V(\Gamma)}x^{D(u,v)}$ \cite{shahkoohi2011polynomial}, $\Xi(\Gamma, x)=\sum_{u \in V(\Gamma)}deg_\Gamma(u)x^{ecc(u)}$ and $\Theta(\Gamma, x)=\sum_{u \in V(\Gamma)}x^{ecc(u)}$ \cite{dovslic2011eccentric} are called the \textit{detour index}, \textit{eccentric connectivity }and \textit{total eccentricity polynomials} , respectively. The first derivative of $D(\Gamma, x)$ at $1$ is called the \textit{detour index} of the graph $\Gamma$, and is denoted by $dd(\Gamma)$. 
		
	A graph is called \textit{regular} if all of its vertices have the same degree and a graph is called an \textit{$n-$regular} if all of its vertices have degree $n$ \cite{bondy1976graph}. While the \textit{chromatic number} of a graph $\Gamma$, denoted by $\chi (\Gamma)$, is the minimum number $c$ for which is $c-$vertex colorable \cite{Godsil}. 
	The \textit{clique number} of a graph $\Gamma$, denoted by $\omega(\Gamma)$, is the size of the largest complete subgraph of $\Gamma$. A \textit{vertex cover} of a graph $\Gamma$ is a subset $S$ of $V(\Gamma)$ such that every edge of $\Gamma$ has at least one vertex in $S$. The minimum size of a vertex cover is denoted by $\tau(\Gamma)$ \cite{pemmaraju2003computational}.
	A non-empty set $S$ of $V(\Gamma)$ is called \textit{independent} if no two elements of $S$ are adjacent in $\Gamma$. The \textit{independent number} is the cardinality of a maximum independent set of a graph $\Gamma$ and is denoted by $\alpha(\Gamma)$ \cite{bondy1976graph}. Let $\Gamma$ be a graph. The \textit{independent polynomial} was defined in \cite{indpoly} as follows:
	$I(\Gamma, x)=\sum_{i=0}^{\alpha(\Gamma)}s_ix^i$; where $s_i$ is the number of independent sets of $\Gamma$ of cardinality $i$. While, in \cite{vcpoly}, the \textit{vertex-cover polynomial} is defined as $\psi(\Gamma, x)=\sum_{i=0}^{n(\Gamma)}c_ix^i$; where $c_i$ is the number of vertex covers of $\Gamma$ of cardinality $i$.
	
	For an integer $k\ge 2$, a graph $\Gamma$ is called \textit{$k-$partite} if $V(\Gamma)$ accept a partition into $k$ classes such that every edge has its ends in different classes, and vertices in the same partition class must not be adjacent. If every two vertices from different partition classes are adjacent, then the graph is called \textit{complete $k-$bipartite graph}, denoted by $K_{r_1, r_2, \cdots, r_k}$.
	
	Let $\Gamma$ be a graph. Suppose $W=\{w_1, w_2, \cdots, w_k\}$ is a subset of $V(\Gamma)$. The \textit{representation} of a vertex $v$ of $\Gamma$ is the $k$-vector $r(v|W)=(d(v, w_1), d(v, w_2), \cdots, d(v, w_k))$. If every pair of distinct vertices of $\Gamma$ have distinct representations with respect to $W$, then $W$ is called a \textit{resolving set} for $\Gamma$. The cardinality of a minimum resolving set for $\Gamma$ is called the \textit{metric dimension} of $\Gamma$, denoted by $\beta(\Gamma)$ \cite{metdim}. The \textit{resolving polynomial} of a graph $\Gamma$, denoted by $\beta(\Gamma, x)$, is defined by $\beta(\Gamma, x)=\sum_{i=\beta(\Gamma)}^{n(\Gamma)}r_ix^i$, where $r_i$ is the number of resolving sets for $\Gamma$ of cardinality $i$. The sequence $(r_{\beta(\Gamma)}, r_{\beta(\Gamma)+1}, \cdots, r_{n{\Gamma}})$ is called the \textit{resolving sequence}. The set of all distinct roots of $\beta(\Gamma, x)$ is denoted by $Z(\beta(\Gamma, x))$. 

	The group $U_{6n}$, of order $6n$, is defined by 
	$$U_{6n}=\langle a, b \;|\; a^{2n}=b^3=1, a^{-1}ba=b^{-1} \rangle$$ for $n \ge 1$ with center $Z(U_{6n})=\langle a^2 \rangle$ \cite{james}. Throughout this paper, the elements of $U_{6n}-Z(U_{6n})$ are partitioned into four disjoint sets according to centralizers of elements, see Lemma \ref{cntlzr}, as follows: $\Omega_1=\{a^{2r+1}\,:\,0\leq r\leq n-1\}$, $\Omega_2=\{a^{2r+1}b\,:\,0\leq r\leq n-1\}$, $\Omega_3=\{a^{2r+1}b^2\,:\,0\leq r\leq n-1\}$, and $\Omega_4=\{a^{2r}b^k\,:\,0\leq r\leq n-1,\,k=1,2\}$. It is clear that $|\Omega_1|=|\Omega_2|=|\Omega_3|=n$, and $|\Omega_4|=2n$.
	
	This paper consists of four sections. In the introduction section, some necessary definitions and notations are presented. In Section 2, some properties of the non-commuting graph $\Gamma({U_{6n}})$ of the group $U_{6n}$ are studied. In Section 3, we find the general formula of the resolving polynomial of the non-commuting graph of the group $U_{6n}$. In Section 4, we find the detour index, eccentric connectivity, total eccentricity and independent polynomials of the non-commuting graph $\Gamma({U_{6n}})$.

	\section{Some basic properties of the non-commuting graph $\Gamma(U_{6n})$}
	This section contains some lemmas on non-commuting graph that are used to obtain some important results that follow.
		
	\medskip
	
	We begin with the following lemma, which can be found in \cite{mirzargar2012some}.
	 
	\begin{lem}\label{cntlzr}
		For the group $\U$, and $0\leq r\leq n-1$, we have the following
		\begin{enumerate}
			\item $Z(\U)=\langle a^2\rangle$,
			\item $C_{U_{6n}}(a^{2r+1})=\la a\ra, $
			\item $C_{U_{6n}}(a^{2r+1}b)=\la a^2\ra\cdot\la\{a^{2s+1}b:0\leq s\leq n-1\}\ra$, 
			\item $C_{U_{6n}}(a^{2r+1}b^2)=\la a^2\ra\cdot\la\{a^{2s+1}b^2:0\leq s\leq n-1\}\ra$, 
			\item $C_{U_{6n}}(a^{2r}b)=\la a^2\ra\cdot\la\{a^{2s}b,a^{2s}b^2:0\leq s\leq n-1\}\ra$.
		\end{enumerate}
	\end{lem}
	
	\medskip
	
	The following useful lemma is used in calculating the degree of vertices in $\Gamma({U_{6n}})$, which can be found in \cite{abdollahi2006non}.
	 	
	\begin{lem}\label{ddeg}
			Let $G$ be a non-abelian finite group and let $x$ be a vertex of $\Gamma(G)$. Then $\deg (x)=|G|-|C_G(x)|$, where $C_G(x)$ is the centralizer of the element $x$ in $G$. 
		\end{lem}
		
	\medskip
	
	The above lemmas lead to the following
	\begin{cor}\label{deg}
		Let $n\geq 1$ be an integer and let $\Gamma=\Gamma({U_{6n}})$. Then, for $0\leq r\leq n-1$ and $k=1,2$, we have
		\begin{enumerate}
			\item $\deg_\Gamma(a^{2r+1})= 4n,$
			\item $\deg_\Gamma(a^{2r+1}b^k)= 4n,$
			\item $\deg_\Gamma(a^{2r}b^k)= 3n.$
		\end{enumerate}
	\end{cor}
	
	\medskip
	
	\begin{thm}\label{EN}
		For $n\geq 1,$ $\Gamma=\Gamma({U_{6n}})$, the number of edges of the graph $\Gamma$ is $|E(\Gamma)|=9n^2$.
	\end{thm}
	\begin{proof}
		From Corollary \ref{deg}, we have $|E(\Gamma)|=\frac{1}{2}\sum_{x\in\Gamma} \deg(x)=\frac{1}{2}\left(12n^2+6n^2\right)=9n^2$. 
	\end{proof}
	
	\medskip
	
	\begin{thm}\label{Kpartite}
		For $n\geq 1$, let $\Gamma=\Gamma({U_{6n}})$ and $\Omega$ is a subset of $U_{6n}$. Then $\Gamma=K_{n,n,n,2n}$ if and only if $\Omega=\Omega_1\cup\Omega_2\cup\Omega_3\cup\Omega_4$.
	\end{thm}
	\begin{proof} 
		Assume $\Omega=\Omega_1\cup\Omega_2\cup\Omega_3\cup\Omega_4$. Then $C_\Omega(a^{2r+1})=\Omega_1, C_\Omega(a^{2r+1}b)=\Omega_2, C_\Omega(a^{2r+1}b^2)=\Omega_3$ and $C_\Omega(a^{2r}b^k)=\Omega_4$ for $k=1,2$, so $\Gamma=K_{n,n,n,2n}$. Conversely, assume $\Gamma=K_{n,n,n,2n}$, then by Corollary \ref{deg}, $\Omega=\Omega_1\cup\Omega_2\cup\Omega_3\cup\Omega_4$.
	\end{proof}
	
	In \cite{bondy1976graph}, the relation between the independent number and the vertex cover with number of vertices has been given as in the following lemma.
	 
	\begin{lem}\label{a+b=n}
		Let $\Gamma$ be a graph. Then $\alpha(\Gamma)+\tau(\Gamma)=n(\Gamma)$.
	\end{lem}
	
	\medskip
	
	\begin{thm}\label{alpha}
		For the graph $\Gamma=\Gamma({U_{6n}})$, $\alpha(\Gamma)=2n$.
	\end{thm}
	\begin{proof}
		From Lemma \ref{cntlzr} and Theorem \ref{Kpartite}, one can see that $\Omega_4$ is the largest  part of the 4-partite graph $K_{n,n,n,2n}$. Thus, $\alpha(\Gamma)=2n.$
	\end{proof}
	
	\medskip
	
	\begin{cor}
		For the graph $\Gamma=\Gamma({U_{6n}})$, $\tau(\Gamma)=3n$.
	\end{cor}
	\begin{proof}
		The result follows from Lemma \ref{a+b=n} and Theorem \ref{alpha}.
	\end{proof}
	
	\medskip
	
	\begin{thm}
		For the graph $\Gamma=\Gamma({U_{6n}})$, we have $\chi(\Gamma)=\omega(\Gamma)=4.$
	\end{thm}
	\begin{proof}
		By Theorem \ref{Kpartite}, the clique of $\Gamma$ can only contain one vertex from each $\Omega_i(i=1,...,4)$. Therefore, $\omega(\Gamma)=4$. Since the group $U_{6n}$ is an AC-group, then $\chi(\Gamma)=\omega(\Gamma)$.
	\end{proof}
	\medskip
	
	\begin{thm}\label{cg}
		Let $\Gamma=\Gamma({U_{6n}})$, and $\Omega=\Omega_1\cup \Omega_2\cup \Omega_3\cup \Omega_4$. There exist no subset $S$ of $\Omega$ such that $\Gamma=C_5$.  
	\end{thm}
	\begin{proof}
		Suppose that $\Gamma=C_5$. Then at least two vertices, say $v_1$ and $v_2$, on $\Gamma$ belong to some $\Omega_i$ for $i \in \{1, 2, 3, 4\}$. There are three cases to consider. \textbf{Case 1.} If the other three vertices belong to $\Omega_j$, $j \in \{1, 2, 3, 4\}$ and $j \neq i$, then $\Gamma=K_{2, 3}$, which is contradiction. \textbf{Case 2.} If two of the other three vertices belong to the same $\Omega_j$ and the third one belongs to $\Omega_k$, where $j \neq i \neq k$, then $\Gamma=K_{2, 2, 1}$, which is also a contradiction. \textbf{Case 3.} If each of the other three vertices belongs to a different $\Omega_j$, $j \in \{1, 2, 3, 4\}$ and $j \neq i$, then $\Gamma=K_{2, 1, 1, 1}$, which is again a contradiction.   
	\end{proof}
	
	\medskip
	
	\begin{thm}
		Let $\Gamma=\Gamma({U_{6n}})$, and $\Omega$ is a subset of $V(\Gamma)$. Then $\Gamma\neq P_k$ for $k\geq 4$.
	\end{thm}
	\begin{proof}
		For $k < 4$, we show that, $\Gamma=P_2$ or $P_3$.\\
		\textbf{Case 1}. $\Omega=\{x,y\}$ where $x\notin C_{U_{6n}}(y)$, then $\Gamma=P_2$.\\
		\textbf{Case 2.} Let $\Omega=\{x,y,z\}$ where $x\notin C_{U_{6n}}(y)$. If $z\notin C_{U_{6n}}(y)$ and $z\notin C_{U_{6n}}(y)$ then $\Gamma=C_3$. But if $z$ is either in $C_{U_{6n}}(x)$ or $C_{U_{6n}}(y)$ then $\Gamma=P_3$.\\
		If we add one more element, say $w$, to $\Omega$ then we have two possibilities; either $w\notin C_{U_{6n}}(x)$ and $w\notin C_{U_{6n}}(y)$, and this implies that there will be edges $d\sim x$ and $d\sim y$, which means $\Gamma\neq P_4$, or $w$ is in the centralizer of one of them, in this case say $w\in C_{U_{6n}}(x)$, implies that there will be an edge $d\sim y$, which again means $\Gamma\neq P_4$.
	\end{proof}
	
	\medskip
	
	\begin{thm}
			Let $\Gamma=\Gamma({U_{6n}})$, and $\Omega$ is a subset of $V(\Gamma)$. Then $\Gamma$ is $2n-$regular if and only if $\Omega=\Omega_1\cup \Omega_2\cup \Omega_3$.
	\end{thm}
	\begin{proof}
	If the graph $\Gamma$ is $2n-$regular, then every vertex in $\Gamma$ has degree $2n$. By Corollary \ref{deg}, $\Omega=\Omega_1\cup \Omega_2\cup \Omega_3$. Conversely, Let $\Omega=\Omega_1\cup \Omega_2\cup \Omega_3$. By Corollary \ref{deg}, $deg(u)=2n$ for all $u \in \Omega$. Thus, $\Gamma$ is a $2n-$regular graph.
	\end{proof}
	\section{Resolving polynomial of the non-commuting graph $\Gamma(U_{6n})$}
   
    In this section, we find the metric dimension and resolving polynomial of the non-commuting graph $\Gamma(U_{6n})$. 
   
    We start by the following useful lemma about resolving polynomial $\beta(\Gamma, x)$ of a graph $\Gamma$ of order $n$.
    
    \begin{lem}\label{resolv}
    If $\Gamma$ is a connected graph of order $n$, then it has only one resolving set of cardinality $n$, which is $V(\Gamma)$, and $n$ resolving sets of cardinality $n-1$. 
    \end{lem}
    
    \medskip
    
    \begin{thm}\label{dim}
    Let $\Gamma=\Gamma(U_{6n})$ be a non-commuting graph on $U_{6n}$. Then
    $
    \beta(\Gamma)=\left\{
    \begin{tabular}{ll}
    $3$ & \mbox{ if }$n=1$\\
    $5n-4$ & \mbox{ if }$n > 1$.\\
    \end{tabular}
    \right.
    $    
    \end{thm}
    \begin{proof}
    \textbf{Case 1}. When $n=1$. The graph $\Gamma$ is a split graph with 5 vertices such that $V(\Gamma)=K\cup S$, where $K=\{a, ab, ab^2\}$ is the complete part and $S=\{b, b^2\}$ is the independent part. The resolving set for $\Gamma$ of minimal cardinality is $W=\{a, ab, b\}$.
    
    \textbf{Case 2}. When $n > 1$. Since every two distinct vertices $u$ and $v$ are non-adjacent in $\Omega_i$, $i\in \{1, 2, 3, 4\}$, then $\beta(\Gamma)\ge 5n-4$. On the other hand, it is clear that the set $W=\{a^{2r+1},$ $ a^{2r+1}b,$ $ a^{2r+1}b^2,$ $ a^{2r}b^k;$ $ k=1, 2;\;$ $ 0\le r \le n-2\}$ is the resolving set for $\Gamma$ of cardinality $5n-4$. This implies that $\beta(\Gamma)\le 5n-4$  
    \end{proof}
    
    \medskip
    
    \begin{thm}
    Let $\Gamma=\Gamma(U_{6n})$ be a non-commuting graph on $U_{6n}$. Then
    
$$\beta(\Gamma, x)=
   \begin{cases}
   x^3(x+2)(x+3) & \mbox{ if } n=1,\\
   x^{5n-4}(x+n)^3(x+2n) & \mbox{ if } n>1.
   \end{cases}
$$
    \end{thm}

    \begin{proof}
     There are two cases to consider. \textbf{When $n=1$}. By Theorem \ref{dim}, it is required to find the resolving sequence $(r_3, r_4, r_5)$ of length 3. Since the graph $\Gamma$ is a split graph with 5 vertices where its complete part consists of 3 vertices and the independent part consists of 2 vertices, then 
     $r_3={{2}\choose{1}}{{3}\choose{2}}=6$.\\
     By Lemma \ref{resolv}, $r_4=5$ and $r_5=1$.\\
     \textbf{When $n>1$}. By Theorem \ref{Kpartite}, the graph $\Gamma$ is 4-partite. By Theorem \ref{dim}, we need to find the resolving sequence $(r_{5n-4}, r_{5n-3}, r_{5n-2}, r_{5n-1}, r_{5n})$ of length 5.\\
     For $r_{5n-4}$: By Theorem \ref{Kpartite} and by the multiplication principal,
     $$r_{5n-4}={{n}\choose{n-1}}{{n}\choose{n-1}}{{n}\choose{n-1}}{{2n}\choose{2n-1}}=2n^4.$$
     
     For $r_{5n-3}$: It is required to count all the resolving sets for $\Gamma$ of cardinality $5n-3$. There are four cases, in the first case, ${{n}\choose{n}}{{n}\choose{n-1}}{{n}\choose{n-1}}{{2n}\choose{2n-1}}$; in the second case, ${{n}\choose{n-1}}{{n}\choose{n}}{{n}\choose{n-1}}{{2n}\choose{2n-1}}$; in the third case, ${{n}\choose{n-1}}{{n}\choose{n-1}}{{n}\choose{n}}{{2n}\choose{2n-1}}$; and in the fourth case, ${{n}\choose{n-1}}{{n}\choose{n-1}}{{n}\choose{n-1}}{{2n}\choose{2n}}$ possible resolving sets of cardinality $5n-3$. By the addition principal, $r_{5n-3}=7n^3$. 
     
     For $r_{5n-2}$: We need to count all the resolving sets for $\Gamma$ of cardinality $5n-2$. Again, we have six cases to consider, 
     in the first case, ${{n}\choose{n}}{{n}\choose{n}}{{n}\choose{n-1}}{{2n}\choose{2n-1}}$; 
     in the second case, ${{n}\choose{n}}{{n}\choose{n-1}}{{n}\choose{n}}{{2n}\choose{2n-1}}$; 
     in the third case, ${{n}\choose{n}}{{n}\choose{n-1}}{{n}\choose{n-1}}{{2n}\choose{2n}}$; 
     in the fourth case, ${{n}\choose{n-1}}{{n}\choose{n}}{{n}\choose{n}}{{2n}\choose{2n-1}}$; 
     in the fifth case, ${{n}\choose{n-1}} {{n}\choose{n}} {{n}\choose{n-1}}$ ${{2n}\choose{2n}}$; 
     and in the sixth case, ${{n}\choose{n-1}}{{n}\choose{n-1}}{{n}\choose{n}}{{2n}\choose{2n-1}}$; 
     possible resolving sets of cardinality $5n-2$. By the addition principal, $r_{5n-2}=9n^2$. 
     By Lemma \ref{resolv}, $r_{5n-1}=5n$ and $r_{5n}=1$.    
    \end{proof}

    \medskip
    
    \begin{cor}
    Let $\Gamma=\Gamma(U_{6n})$ be a non-commuting graph on $U_{6n}$. Then for $n=1$, $Z(\beta(\Gamma, x))=\{0, -2, -3\}$, and for all $n>1$, $Z(\beta(\Gamma, x))=\{0, -n, -2n\}$.
    \end{cor}

  \section{Some polynomials of the non-commuting graph $\Gamma(U_{6n})$}
  
  In this section some properties of non-commuting graphs of $U_{6n}$ are explored, namely the detour index, the eccentric connectivity, the total eccentricity and the independent polynomials.
  
  	\begin{lem}\label{lem41}
  		Let $\Gamma({U_{6n}})$ be a non-commuting graph of $U_{6n}$. Then, $D(u, v)=5n-1$ for any $u, v \in \Gamma({U_{6n}})$.
  	\end{lem}
  	\begin{proof}
  	 From Theorem \ref{Kpartite}, one can see that no two vertices in $\Omega_i$ are adjacent, and each vertex in $\Omega_i$ is adjacent to every vertex in $\Omega_j$ for $i \neq j$ and $i, j\in \{1, 2, 3, 4\}$. Then for all $u, v \in \Omega$, where $\Omega=\Omega_1 \cup \Omega_2\cup \Omega_3\cup \Omega_4$, there is a $u-v$ path of length $5n-1$.
  	\end{proof}
  	
  	\medskip
  	
  	\begin{thm}\label{thm24}
  	Let $\Gamma({U_{6n}})$ be a non-commuting graph of $U_{6n}$. Then, $D(\Gamma({U_{6n}}), x)=\frac{5n(5n-1)}{2}x^{5n-1}$.
  	\end{thm}
  	\begin{proof}
  	We have that $|\Gamma({U_{6n}})|=5n$. Then there are ${{5n}\choose{2}}=\frac{5n(5n-1)}{2}$ possibilities of choosing distinct pairs of vertices from $\Gamma({U_{6n}})$. By Lemma \ref{lem41}, $D(u, v)=5n-1$ for any distinct pairs of $u, v \in \Gamma({U_{6n}})$. Thus, $D(\Gamma({U_{6n}}), x)=\sum_{\{u,v\}}x^{D(u,v)}={{5n}\choose{2}}x^{5n-1}=\frac{5n(5n-1)}{2}x^{5n-1}$. 
  	\end{proof}
  
     \medskip
   Theorem \ref{thm24} leads to the following result.
  \begin{cor}
  Let $\Gamma({U_{6n}})$ be a non-commuting graph of $U_{6n}$. Then, $dd(\Gamma({U_{6n}}))=\frac{5n(5n-1)^2}{2}$.
  \end{cor}
    
  \medskip
  
  	\begin{lem}\label{lem43}
  			Let $\Gamma({U_{6n}})$ be a non-commuting graph of $U_{6n}$. Then, $ecc(u)=2$ for every $u \in \Omega$, where $\Omega=\Omega_1 \cup \Omega_2\cup \Omega_3\cup \Omega_4$.
  	\end{lem}
  	\begin{proof}
  	In $\Omega_i$, there is no edge between any pair of distinct vertices, for $i \in \{1, 2, 3, 4\}$. Furthermore, each vertex in $\Omega_i$ is adjacent to every vertex in $\Omega_j$, for $i \neq j$ and $i, j\in \{1, 2, 3, 4\}$. Then the maximum distance between any vertex in $\Omega_i$ and other vertices in $\Omega$ is $2$. Therefore, $ecc(u)=2$, for every $u \in \Omega$.  	
  	\end{proof}
  	
  	\medskip
  	
  	\begin{thm}
  	Let $\Gamma=\Gamma({U_{6n}})$ be a non-commuting graph of $U_{6n}$ and $\Omega=\Omega_1 \cup \Omega_2\cup \Omega_3\cup \Omega_4$. Then,
  	\begin{enumerate}
  	\item $\Theta(\Gamma, x)=5nx^2$.
  	\item $\Xi(\Gamma, x)= 18n^2x^2$.
  	\end{enumerate}
  	\end{thm}
  	\begin{proof}
  	Since the graph $\Gamma$ has $5n$ vertices, then
  	\begin{enumerate}
  	\item  By Lemma \ref{lem43} $ecc(u)=2$, for every $u \in \Omega$, so $\Theta(\Gamma, x)=\sum_{u \in V(\Gamma)}x^{ecc(u)}=5nx^2$.
  	\item By Corollary \ref{deg}, $3n$ vertices in $\Omega_1 \cup \Omega_2\cup \Omega_3$ are of degree $4n$ and $2n$ vertices in $\Omega_4$ are of degree $3n$ and from Lemma \ref{lem43} we see that $\Xi(\Gamma, x)=\sum_{u \in V(\Gamma)}deg_\Gamma(u)x^{ecc(u)}=(3n(4n)+2n(3n))x^2=18n^2x^2$.
  	\end{enumerate}
  	\end{proof}
  	
    \begin{thm}\label{thm46}
  	Let $\Gamma=\Gamma({U_{6n}})$ be a non-commuting graph of $U_{6n}$. Then the independent polynomial is as follows, $$I(\Gamma; x)=1+\sum_{k=1}^{n}\left({{2n}\choose{k}}+3{{n}\choose{k}}\right)x^k+\sum_{k=n+1}^{2n}{{2n}\choose{k}}x^k.$$
  	\end{thm}
	\begin{proof}
		By Theorem \ref{alpha}, the independent number of $\Gamma$, $\alpha(\Gamma)=2n$. Then 		$I(\Gamma;x)=\sum_{k=0}^{2n}s_kx^k.$ It is easy to see that $s_0=1$ since the only independent set of cardinality zero is the empty set. Moreover, we have three independent sets, $\Omega_1, \Omega_2$ and $\Omega_3$, each of cardinality $n$ and one independent set, $\Omega_4$, of cardinality $2n$. 
		Thus, there are $s_k=3{{n}\choose{k}}+{{2n}\choose{k}}$ possibilities of independent sets of cardinality $k$ for $1\leq k\leq n$, and $s_k={{2n}\choose{k}}$ possibilities of independent sets of cardinality $k$ for $n<k\leq 2n$. Then the result follows. 
	\end{proof}
  	
  	\medskip
  	
  	\begin{cor}
  	Let $\Gamma=\Gamma({U_{6n}})$ be a non-commuting graph of $U_{6n}$. Then, the vertex-cover polynomial is as follows,
 $$\psi(\Gamma; x)=x^{5n}+\sum_{k=1}^{n}\left({{2n}\choose{k}}+3{{n}\choose{k}}\right)x^{5n-k}+\sum_{k=n+1}^{2n}{{2n}\choose{k}}x^{5n-k}.$$
  	\end{cor}
  	\begin{proof}
  	The result comes from Theorem \ref{thm46} and the fact that $\psi(\Gamma, x)=x^{n(\Gamma)}I(\Gamma, x^{-1})$ \cite{akbari}.
  	\end{proof}
 

\section*{Conclusion}

In this paper, some properties of the non-commuting graph of the group $\U$ is presented. The general formula of the resolving polynomial of the non-commuting graph of the group $U_{6n}$ are provided. In the last section of this paper, we also provided the detour index, eccentric connectivity, total eccentricity and independent polynomials of non-commuting graphs on $U_{6n}$.

\section*{Acknowledgment}

The third author would like to appreciate Universiti Teknologi Malaysia for the research grant with number 20H70 under Fundamental Research Grant scheme and the fourth author would like to acknowledge his postdoctoral fellowship.


\end{document}